\documentclass[12pt, reqno]{amsart}
\usepackage{amsmath, amsthm, amscd, amsfonts, amssymb, graphicx, color}
\usepackage[bookmarksnumbered, colorlinks, plainpages]{hyperref}

\makeatletter \oddsidemargin.9375in \evensidemargin \oddsidemargin
\def\Speaker{$^{*}$\protect\footnotetext{$^{*}$ S\lowercase{peaker.}}}
\def\authorsaddresses#1{\dedicatory{#1}}
\newtheorem{theorem}{Theorem}[section]

\theoremstyle{definition}
\newtheorem{definition}[theorem]{Definition}

\theoremstyle{remark}

\numberwithin{equation}{section}

\begin{document}
\setcounter{page}{1}

\noindent {\footnotesize The Extended Abstracts of \\
The 4$^{\rm th}$ Seminar on Functional Analysis and its Applications\\
2-3rd March 2016, Ferdowsi University of Mashhad, Iran}\\[1.00in]

\title[Relative Reproducing Kernels in Vector-Valued Hilbert and Banach Spaces]{Relative Reproducing Kernels in Vector-Valued Hilbert and Banach Spaces}

\author[Ebadian, Hashemi Sababe]{Ali Ebadian$^1$,\Speaker Saeed Hashemi Sababe$^2$}

\authorsaddresses{$^1$ Department of Mathematics, Payame Noor University (PNU), Iran;\\
Ebadian.ali@gmail.com\\
\vspace{0.5cm} $^2$ Department of Mathematics, Payame Noor University (PNU), Iran;\\
Hashemi\underline{~}1365@yahoo.com}


\begin{abstract}
This paper is devoted to the study of vector valued reproducing
kernel Hilbert spaces. We focus on reproducing kernels in vector-valued reproducing kernel Hilbert spaces.
In particular we extend reproducing kernels to relative reproducing kernels and prove some theorems in this subject.
\end{abstract}

\maketitle


\section{Introduction}

The purpose of this paper is to establish the notion of vector-valued relative reproducing kernels in Hilbert and Banach spaces. In learning theory, reproducing kernel Hilbert spaces (RKHS) and reproducing kernel Banach spaces (RKBS) are important tools for designing learning algorithms. The basic object is a Hilbert or Banach space of functions $f$ from a set $X$ into a normed vector space $\mathcal{Y}$ with the property that, for any $x \in X, \vert f(x)\vert \leq C_x \vert f \vert$  for a positive constant $C_x$ independent of $f$. In continues, we briefly review basic definitions. More details can be found in \cite{1,2,3,4,5}. \par
Given a set $X$ and a normed vector space $\mathcal{Y}$, a map $K : X \times X \rightarrow \mathcal{L}(\mathcal{Y})$ is called a $\mathcal{Y}$-reproducing kernel if
\begin{equation}
\sum_{i,j=1}^n \langle K(x_i,x_j)y_j,y_i\rangle  \geq 0 \label{1}
\end{equation}
or any $x_1, \dots , x_n$ in $X$, $y_1, \dots , y_n$ in $\mathcal{Y}$ and $n \geq 1$. Given $x \in X$, $K_x : \mathcal{Y} \rightarrow \mathcal{F}(X,\mathcal{Y})$ denotes the linear operator whose action on a vector $y \in \mathcal{Y}$ is the function $K_{xy} \in \mathcal{F}(X,\mathcal{Y})$ defined by
\begin{equation}
(K_{xy})(t) = K(t, x)y \qquad t \in X. \label{2}
\end{equation}
Given a $\mathcal{Y}$-reproducing kernel $K$, there is a unique Hilbert space $\mathcal{H}_K \subset \mathcal{F}(X, \mathcal{Y})$ satisfying
\begin{eqnarray}
&K_x \in L(\mathcal{Y}, \mathcal{H}_K)\qquad  x \in X \\
&f(x) = K_x^∗f \qquad  x \in X, f \in \mathcal{H}_K, \label{3}
\end{eqnarray}
where $K_x^* : \mathcal{H}_K \rightarrow \mathcal{Y}$ is the adjoint of $K_x$. The space $\mathcal{H}_K$ is called the reproducing kernel Hilbert space associated with $K$, the corresponding scalar product and norm are denoted by $\langle ., . \rangle_K$ and $\Vert . \Vert_K$,
respectively. As a consequence of $(\ref{3})$, we have that
\begin{eqnarray}
&K(x, t) = K_x^∗K_t \qquad  x, t \in X \\
&\mathcal{H}_K = \overline{span} \{K_{xy} \vert x \in X, y \in Y \} .
\end{eqnarray}
Let $\mathcal{H}$ be a Hilbert space of functions defined on the set $X$. We say that it is a relative reproducing kernel Hilbert space if there exists a function $M_{x,y}$ from $X \times X$ into $\mathcal{H}$ such that
\begin{equation}
 F(x) - F(y) = \langle F, M_{x,y},\mathcal{H}\rangle \qquad \forall x, y \in X , \forall F \in \mathcal{H}.
\end{equation} \par
A semi-inner product on a Banach space
$V$ is a function from $V \times V$ to $\mathbb{C}$, denoted by $[ ., .]_V$, such that for all $u, v, w \in V$ and $\alpha , \beta \in \mathbb{C}$
\begin{enumerate}
\item (linearity ) $[\alpha f + \beta g, h]_V = \alpha [f, h]_V + \beta [g, h]_V$;
\item (positivity) $[f, f]_V > 0$ for $f \neq 0$;
\item (conjugate homogeneity) $[f, \alpha g]_V = \alpha [f, g]_V$;
\item (Cauchy-Schwartz inequality) $\vert  [f, g]_V\vert \leq  [f, f]_V^{1/2}[g, g]_V^{1/2}$.
\end{enumerate}
A semi-inner product $[., .]_V$ on $V$ is said to be compatible if
\begin{equation}
[f, f]_V^{1/2} = \Vert f\Vert_V \qquad \forall f \in V,
\end{equation}
where $\Vert . \Vert_V$ denotes the norm on $V$. Every Banach space has a compatible semi-inner product.
Let $[., .]_V$ be a compatible semi-inner product on $V$. Then one sees by the Cauchy-Schwartz inequality
that for each $f \in \mathcal{B}$, the linear functional $f^*$ on $V$ defined by
\begin{equation}
f^*(g) := [g, f]_V, \qquad g \in V
\end{equation}
is bounded on $V$. In other words, $f^*$ lies in the dual space $\mathcal{B}^∗$ of $\mathcal{B}$. Moreover, we have
\begin{equation}
\Vert f^*\Vert_{V^*} = \Vert f\Vert_V
\end{equation}
Introduce the duality mapping $\mathcal{I}_V$ from $V$ to $V^∗$ by setting
\begin{equation}
\mathcal{I}_V(f) := f^∗,\qquad f \in V.
\end{equation}
Let $\Lambda$ be a Banach space. A space $\mathcal{B}$ is called a Banach space of $\Lambda$-valued functions on $X$ if it consists of certain functions from $X$ to $\Lambda$ and the norm on $\mathcal{B}$ is compatible with point evaluations in the sense that
\begin{equation}
\Vert f\Vert_B = 0 \Longleftrightarrow  f(x) = 0 \quad \forall x \in X.
\end{equation}
We call $\mathcal{B}$ a $\Lambda$-valued RKBS on $X$ if both $\mathcal{B}$ and $\Lambda$ are uniform and $\mathcal{B}$ is a Banach space of functions from $X$ to $\Lambda$ such that for every $x \in X$, the point evaluation $\delta_x : \mathcal{B} \rightarrow \lambda$ defined by
\begin{equation}
\delta_x(f) := f(x), \qquad f \in \mathcal{B},
\end{equation}
is continuous from $\mathcal{B}$ to $\Lambda$.\\
\section{Main Results}
\begin{definition}
Let $X$ be set an arbitrary set, $\mathcal{Y}$ a normed vector space and $\mathcal{H}$ a vector-valued reproducing kernel Hilbert space with $K$ as its reproducing kernel satisfying in properties $(\ref{1},\ref{2})$. We say that it is a vector-valued relative reproducing kernel Hilbert space (RRKHS) if there exists a function $M_{x,y}$ such that
\begin{eqnarray}
&M_{x,y} \in \mathcal{L}(\mathcal{Y},\mathcal{H}) \qquad \forall x,y\in X \label{4}\\
&f(y)-f(x)=M_{x,y}^*f \qquad \forall x,y\in X, \forall f\in \mathcal{H}
\end{eqnarray}
where $M_{x,y}^* : \mathcal{H} \rightarrow \mathcal{Y}$ is the adjoint of $M_{x,y}$.
\end{definition}
\begin{definition}
Let $\mathcal{B}$ be a Banach space of $\Lambda$-valued functions on $X$. We call $\mathcal{B}$ a $\Lambda$-valued relative reproducing kernel Banach space (RRKBS) on $X$ if both $\mathcal{B}$ and $\Lambda$ are uniform and $\mathcal{B}$ is a Banach space of functions from $X$ to $\Lambda$ such that for every $x \in X$, the function $\zeta_{x,y} : \mathcal{B} \rightarrow \lambda$ defined by
\begin{equation}
\zeta_{x,y}(f) := f(y)-f(x), \qquad \forall f \in \mathcal{B}, \label{6}
\end{equation}
is continuous from $\mathcal{B}$ to $\Lambda$.
\end{definition}
\begin{theorem}\label{t1}
The function $M_{x,y}$ in (\ref{4}) is unique and satisfies
\begin{equation}
 M_{x_1,x_2}(t) + M_{x_2,x_3}(t) = M_{x_1,x_3}(t)\qquad \forall x_1, x_2, x_3, t \in X. \label{5}
\end{equation}
\end{theorem}
\begin{proof}
The uniqueness of the function $M_{x_1,x_2}$ follows from Riesz’ representation theorem. This uniqueness and following equality imply (\ref{5}).
\begin{equation}
f(x_1) - f(x_3) = f(x_1) - f(x_2) + f(x_2) - f(x_3)
\end{equation}
\end{proof}
\begin{theorem}
Previews theorem also holds in RRKBS but the function is not necessarily unique in these spaces.
\end{theorem}
\begin{proof}
similarly to theorem \ref{t1}.
\end{proof}
Obvously the Riesz’ representation theorem does not hold in general Banach spaces but with adding some conditions, this theorem can be hold. So in general case, the above uniqueness does not hold. \begin{theorem}
Let $\mathcal{B}$ be a $RRKBS$ on a locally compact Hausdorf Banach space $\Lambda$. Then the function defined in (\ref{6}) is unique.
\end{theorem}
\begin{proof}
In this case, a version of Riesz’ representation theorem holds and it implies the uniqueness. see Theorem 1 in \cite{3}.
\end{proof}
\begin{theorem}
Let $\mathcal{H}$ be a Hilbert space and there exist functions $K$ and $M$ with reproducing and relative reproducing property, respectively, then
\begin{equation}
\mathcal{H}_M \subseteq \mathcal{H}_K \subseteq \mathcal{H}
\end{equation}
\end{theorem}
\bigskip
\section*{Acknowledgement} Research supported in part by Kavosh Alborz Institute
of Mathematics and Applied Sciences and was accomplished while the second author was visiting the University of Iowa.

\hspace{1in}

\bibliographystyle{amsplain}

\end{document}